\documentclass[12pt]{amsart}
\usepackage{amssymb,amsmath,amsthm,fullpage}
\usepackage{pdfsync,color}
\usepackage{enumerate}
\newcommand{\R}{\mathbb R}

\newcommand{\C}{\mathbb C}

\DeclareMathOperator{\Imm}{Im}

\DeclareMathOperator{\Rre}{Re}

\newcommand{\p}{\partial}

\newcommand{\dbarb}{\bar\partial_b}
\newcommand{\Boxb}{\Box_b}

\DeclareMathOperator{\csch}{csch}

\DeclareMathOperator{\spt}{supp}

\newcommand{\SCH}{Schr\"{o}dinger}
\newcommand{\HOL}{H\"{o}lder}

\newcommand{\Ainf}{\ensuremath{A_{\infty}}}
\DeclareMathOperator{\av}{av}

\newtheorem{thm}{Theorem}[section]
\newtheorem{prop}[thm]{Proposition}

\newtheorem{lmm}[thm]{Lemma}

\theoremstyle{definition}
\newtheorem{defn}[thm]{Definition}

\theoremstyle{remark}

\begin{document}

\title{Schr\"odinger Operators With $A_\infty$ Potentials}

\author{Andrew Raich and Michael Tinker}

\thanks{The first author was partially supported by NSF grant DMS-1405100.}
\thanks{The main results were  part of Tinker's Ph.D. thesis which he completed under Raich's supervision.}

\address{Department of Mathematical Sciences, SCEN 309, 1 University of Arkansas, Fayetteville, AR 72701}
\email{araich@uark.edu}

\address{Plano, TX}
\email{michael.tinker@ca.com}

% Key words and phrases:

%% Mathematical classification (2010)

\begin{abstract} We study the heat kernel $p(x,y,t)$ 
associated to the real Schr\"odinger operator $H = -\Delta + V$ on $L^2(\mathbb{R}^n)$,
$n \geq 1$. 
Our main result is a pointwise upper bound on $p$ when the potential
$V \in A_\infty$. In the case that $V\in RH_\infty$, we also prove a lower bound. 
Additionally, we compute $p$ explicitly when $V$ is a quadratic polynomial. 
\end{abstract}

\maketitle

%%%%%%%%%%%%%%%%%%%%%%
%
%				SECTION: MOTIVATION AND RELATED WORK
%
%%%%%%%%%%%%%%%%%%%%%%%
\section{Motivation and Related Work}\label{inspire}

% subsection: real Schrodinger operators
\subsection{Real Schr\"odinger operators}
\SCH\ operators $H = -\Delta + V$ enjoy 
considerable  
interest due to their physical importance. For example, thousands of 
papers have been devoted to the study of 
quantal anharmonic oscillators, which feature $H$ with potential 
$V(x) = x^2 + \lambda x^{2m}$. 
In general, given any $V \geq 0$ in $L_{\textrm{loc}}^1(\mathbb{R}^n)$, 
we can define $H$ as an operator on $L^2(\mathbb{R}^n)$ through 
its Dirichlet form (see \cite{AuBA07}). This leads to an
associated $H$-heat equation
\[
\begin{cases} \p_t u + Hu =0 & \text{in } \R^n\times(0,\infty) \\
\displaystyle \lim_{t\to 0^+} u(x,t) = f(x) & \text{on } \R^n
\end{cases}
\]
where the limit is in $L^2(\mathbb{R}^n)$. General  
solutions to this equation are weak solutions (defined below)
given by integration against a kernel $p(x,y,t)$ called the 
heat kernel of $H$. Namely,
\[
u(x,t) = e^{-tH}f(x) = \int_{\R^n} p(x,y,t)f(y)\, dy,
\]

Nonnegativity of $V$ implies a trivial Gaussian
bound $p(x,y,t) \leq (4\pi t)^{-n/2}e^{-\frac{|x-y|^2}{4t}}$ 
(see \cite{Dav89}). 
Looking closer, there is a diverse 
literature on upper bounds for $p$,  
mirroring the variety 
of interesting potentials to work with. 
A common reference point is Davies' results for $H$ with continuous 
potentials diverging to infinity as $|x| \to \infty$; he established 
the qualitatively sharp bounds
\[
p(x,y,t) \leq c(t)\phi(x)\phi(y)
\]
where $\phi$ is the $L^2$-normalized ground state of $H$, and 
$c(t)$ has an explicit description as $t \to 0$ (again see \cite{Dav89}). 
For classes of potentials that are not as well-behaved, one still hopes to 
prove extra-Gaussian decay in terms of $V$, even if sharp results 
are not attainable. 

% Connection with several complex variables
\subsection{Connection with several complex variables}
Our original motivation was the Kohn Laplacian $\Boxb$ on a class of CR manifolds in $\C^2$.
For a subharmonic function $\phi$, the three-dimensional CR manifold $M$ defined by
$$M  = \{(z,w) \in \C^2 : \Imm w = \phi(z)\}$$
is pseudoconvex, the complex analysis version of convexity.
Since $M$ does not depend on $\Rre w$, we can analyze the Kohn Laplacian $\Boxb$ and the $\Boxb$-heat kernel via a partial Fourier transform in $\Rre w$. This has been carried
out by several mathematicians, e.g., \cite{Chr91, Rai06h, Rai06f, Rai07, Rai12, BoRa13h, Ber96}. If, in addition, $\phi(z) = \phi(\Rre z)$ only depends on $\Rre z$, then Nagel observed
that a partial Fourier transform in $\Imm z$ reduces $\Boxb$ and its associated operators even further \cite{Nag86}. If $\tau$ and $\eta$ are the transform variables of $\Rre w$ and $\Imm z$, respectively,
then
\[
\widehat\Boxb = -\Delta + \phi''\tau + (\eta - \tau\phi')^2,
\]
an operator of the form $H$ with $V = \phi''\tau + (\eta-\tau\phi')^2$. Since $\phi$ is convex (a subharmonic function of one variable is convex), $V\geq 0$ when $\tau\geq 0$. Nagel's observation
about the reduced form of $\dbarb$ and its associated operators when $\phi(z) = \phi(\Rre z)$ has been repeatedly exploited \cite{Has94, HaNaWa10, RaTi15}. Once we have sufficient estimates on
$H$, we will be able to recover information about $\Boxb$, in the same spirit as \cite{Rai12, BoRa13h}.

% subsection: Potentials in $A_\infty$
\subsection{Potentials in reverse H\"older classes and $A_\infty$}
In this note, we focus on $V$
belonging to the Muckenhoupt class $A_\infty$ \cite{Muc72}. Membership in $A_\infty$ 
is equivalent to membership in some reverse \HOL\ class $RH_q$ for $q > 1$, where 
$RH_q$ is defined as follows. 
\begin{defn}\label{rvhold}
For $1 < q \leq \infty$, nonnegative $V \in L_{loc}^{q}(\mathbb{R}^{n})$ belongs to the
 reverse H\"{o}lder class $RH_{q}$ if there exists $C > 0$ such that 
 for all cubes $Q$
 of $\mathbb{R}^{n}$,
 $$
 	\biggl(\frac{1}{|Q|}\int_{Q}V^{q}\,dx\biggr)^{1/q} \leq \frac{C}{|Q|}\int_{Q}V\,dx
$$
where for $q = \infty$ the left hand side is the ess sup over $Q$. 
In particular, $A_{\infty} = \cup_{q > 1}RH_{q}$. 
\end{defn}
Such potentials need not have uniform behavior at infinity, e.g., $V(x_1,x_2,x_3) = 
(x_1x_2x_3)^2$; and may have integrable singularities, e.g., $V(x) = |x|^{-\alpha}$ 
with $\alpha < n$.  See Stein \cite{Ste93} for further details. In \cite{Kur00} 
Kurata proved, for $n \geq 2$ and $V \in RH_q$ with $q \geq n/2$, that 
\begin{equation}\label{whynot}
	p(x,y,t) \leq \frac{c_{0}}{t^{n/2}}e^{-c_{2}\frac{|x-y|^{2}}{t}}
	\exp\left\{-c_{1}(1 + m_{V}(x)^{2}t)^{1/(2(k_{0} + 1))}\right\}
\end{equation}
where $m_{V}(x)$ is a function measuring the effective growth of $V$ near $x$.
This estimate gives a non-sharp order of decay for computable examples such 
as $V = |x|^\alpha$ where $\alpha > 0$; hence its primary interest lies in the 
diversity of potentials to which it applies. The natural question is whether 
we can go further to remove the limitations on $n$ and $q$.  

\section{An Upper Bound for $p$ when $V \in A_\infty$}\label{result}
We provide just such an analogue of \eqref{whynot} for
all $n \geq 1$ and $q > 1$. Specifically, let
\[
\av_{Z_{r}(x)} V = \frac{1}{|Z_{r}(x)|} \int_{Z_{r}(x)} V\, dx
\]
denote the average of $V$ over the cube $Z_r(x)$ centered
at $x$ with side length $r$. Then we establish the following.
% Theorem: main result, upper bounds
\begin{thm}\label{bus} If $V \in A_{\infty}$, the heat kernel of the Schr\"{o}dinger
operator $H = -\Delta + V$ on $L^{2}(\mathbb{R}^{n})$ satisfies
\begin{equation}\label{convenient}
	p(x,y,t) \leq \frac{c_{0}}{t^{n/2}}e^{-c_{2}\frac{|x-y|^2}{t}}
		\exp\left\{-c_{1}m_{\beta}(t\,\av_{Z_{\sqrt{t}}(x)} V)^{1/2}\right\}		
\end{equation}
where $c_{i} > 0$ for $i =0,1,2$  and  $m_{\beta}(x) = x$ for $x \leq 1$ and $m_{\beta}(x) = x^{\beta}$ for $x \geq 1$. 
In particular, if $V \in A_{p}$, then one may take $\beta = \frac{2}{2 + n(p - 1)}$.
\end{thm}

To orient ourselves with how close our result is to being sharp, we provide the explicit heat kernel for 
quadratic polynomial $V$ on $\mathbb{R}$. 
%(Section~\ref{proof} provides a calculation that uses only elementary methods.)
\begin{thm}\label{explicit}
If $V(x)= \sum_{i=0}^2 a_{i}x^{i}$ with $a_{2} > 0$, then the heat kernel $p(x,y,t)$ of the \SCH\ operator $H = -\Delta + V$ on $L^2(\R^n)$ is given by the formula
\begin{align*}
		p(x,y,t) = &\left(\frac{\sqrt{a_{2}} \csch 2\sqrt{a_{2}}t}{2\pi}\right)^{1/2} 
		\, e^{\left(\frac{a_{1}^2}{4a_{2}} - a_{0}\right)t} \\
	&\ \cdot \exp\left\{-\frac{{a_{1}}^2}{4(\sqrt{a_{2}})^{3}}
    (\coth 2\sqrt{a_{2}}t - \csch{2\sqrt{a_{2}}t}) \right\} \\
	&\ \cdot \exp\left\{ -\frac{\sqrt{a_{2}}}{2}\biggl( (x-y)^2\csch{2\sqrt{a_{2}}t } +
			(x^2 + y^2)(\coth{2\sqrt{a_{2}}t} - \csch{2\sqrt{a_{2}}t})\biggr)\right\} \\
	&\ \cdot \exp\left\{-\frac{a_{1}}{2\sqrt{a_{2}}}
		\biggl((x+y)(\coth 2\sqrt{a_{2}}t - \csch{2\sqrt{a_{2}}t})\biggr)\right\}
\end{align*}
for $x,y\in\R^n$ and $t>0$.
\end{thm}

First, recall the asymptotics 
\begin{align*}
	\bullet \quad & \csch(t) \sim t^{-1} \textrm{ as } t \to 0^{+} \textrm{ and } \csch(t) \sim e^{-t} 
	\textrm{ as } t \to +\infty \\
	\bullet \quad &  (\coth(t) - \csch(t)) \sim t \textrm{ as } t \to 0^{+} \textrm{ and } 
	(\coth(t) - \csch(t)) \sim 1 \textrm{ as } t \to +\infty,
\end{align*}
So a sharp upper bound for the above formula is essentially 
\begin{equation}\label{dream}
p(x,y,t) \lesssim \left\{\begin{array}{ll}
			t^{-1/2}e^{-c_{0}\frac{|x-y|^2}{t}}
	\exp\left\{-c_{1}t(x^2 + y^2)\right\} & t \leq 1  \\
			e^{-c_{2}t}\exp\left\{-c_{3}(x^2 + y^2)\right\}& t > 1
		\end{array}\right.
\end{equation}
To compare this to Theorem~\ref{bus}, we rewrite \eqref{convenient} to 
include a decay term in $y$,
\begin{align}
p(x,y,t) &= p(x,y,t)^{1/2}\cdot p(y,x,t)^{1/2} \nonumber \\
 &\lesssim t^{-1/2}e^{-c_{1}\frac{|x-y|^2}{t}}
		\exp\left\{-c_{2}\left[m_{\beta}(t\,\av_{Z_{\sqrt{t}}(x)} V)^{1/2}
		+ m_{\beta}(t\,\av_{Z_{\sqrt{t}}(y)} V)^{1/2}\right]\right\}.	\label{bothxy}
\end{align}
Taking $n=1$ and assuming quadratic $V(x)$, we have
$$
\av_{Z_{\sqrt{t}}(x)}V = \frac{1}{\sqrt{t}}\int_{x-\frac{1}{2}\sqrt{t}}^{x+\frac{1}{2}\sqrt{t}}V(z)\,dz 
 = a_{2}\left(x^2 + \frac{t}{12}\right) + a_{1}x + a_{0}.
$$
Thus we see that \eqref{bothxy} is no sharper than a bound of
$$
p(x,y,t) \lesssim t^{-1/2}e^{-c_{1}\frac{|x-y|^2}{t}}
		\exp\left\{-c_{2}\left[m_{\beta}(t^{1/2}|x| + t)
		+ m_{\beta}(t^{1/2}|y| + t)\right]\right\}.
$$
When all of the three terms $\{\sqrt{t}, |x|, |y|\}$ are small, the Gaussian factor will essentially
determine the size of both the above, and of \eqref{dream}. But when $t > 1$ and say $|x|$
is the dominant term, our upper bound can be no sharper than $\exp\left\{-c |x|^{2\beta}\right\}$, 
while \eqref{dream} will have decay of the order $\exp\left\{-cx^2\right\}$. 
Hence the presence
of the sublinear function $m_{\beta}$ prevents our estimates from being attained.  However, if
$V(x)$ is a strictly positive polynomial, $\beta$ may be taken arbitrarily close to 1 since positive 
polynomials belong to $RH_{\infty}$.

Before moving to the proof of Theorem~\ref{bus}, we also 
take a moment to sketch its strategy,  
which uses ideas of Shen \cite{she95}. 
Fix $y \in \mathbb{R}^n$, and 
look at $p = p(\cdot,y,t)$ in a cylinder $Q \subset \mathbb{R}^n \times
(0,\infty)$. Since $p$ is a weak solution to $(\partial_t + H)u = 0$
in this cylinder,
Moser's work on local boundedness implies $\sup_{Q}p$ is dominated by
its $L^2$ norm over a slightly larger cylinder. Now take  
an increasing sequence of cylinders beginning with $Q$. Given 
appropriate Fefferman-Phong and Caccioppoli type inequalities, 
we can alternate upper bounds
for $p$ in terms of its $L^2$ energy and $L^2$ norm over this sequence
of cylinders. 
Each iteration introduces another factor of a $V$-dependent weight in 
the evolving upper bound. When we conclude the iteration by applying the
Gaussian bound on $p$, our result is extra-Gaussian decay in 
terms of $V$. 

So there are three main ingredients: local boundedness, a Fefferman-Phong
inequality valid for $V \in A_\infty$, and a Caccioppoli
inequality for weak solutions to $(\partial_t + H)u = 0$. We
introduce these items in Subsections~\ref{lb}, \ref{fp}, and \ref{ci}, 
respectively; then combine them for the proof in Section~\ref{proof}.

\subsection{Local boundedness}\label{lb}
Denote cylinders in $\mathbb{R}^n\times(0,\infty)$ with the notation,
$$Q_r(x_{0}, t_{0}) 
	=  B(x_{0},r) \times I_{t_{0},r} = B(x_{0},r) \times (t_{0} - r^{2}, t_{0})$$
And consider \emph{weak solutions} of $(\partial_{t} + H)u = 0$ as below,
\begin{defn}\label{weakly}A real-valued function $u(x,t)$ is a weak solution to
$(\partial_{t} + H)u = 0$ in $Q_r(x_{0}, t_{0})$ if $u \in L^{\infty}(L^{2}(B(x_{0},r)); I_{t_{0}, r}) \cap L^{2}(H^{1}(B(x_{0},r)); I_{t_{0},r})$
satisfies
\begin{multline}\label{ws}
		\int_{B(x_{0},r)}u(x,t)\phi(x,t)\,dx - \int_{t_{0}-r^{2}}^{t}\int_{B(x_{0},r)}
			u(x,s)\partial_{s}\phi(x,s)\,dx\,ds \\
			+ \int_{t_{0}-r^{2}}^{t}\int_{B(x_{0},r)} \bigl(\nabla u(x,s) \cdot \nabla \phi(x,s)
			+ V(x) u(x,s)\phi(x,s)\bigr)\,dx\,ds = 0
\end{multline}
for $t_{0}-r^{2} < t \leq t_{0}$ and for every $\phi(x,s) \in \mathcal{C}$, where
$$\mathcal{C} = \{\phi \in L^{2}(H^{1}(B(x_{0},r)); I_{t_{0},r}) \textrm{ and }
	\partial_{s}\phi \in L^{2}(L^{2}(B(x_{0},r); I_{t_{0},r}); \phi(x, r_{0}-r^{2})) = 0\}$$
\end{defn} 
In this setting Moser established the following fundamental result, 
which applies to $p(\cdot,y,t)$ because it is a weak solution of 
$(\partial_{t} + H)u = 0$ on every such cylinder (see \cite{Bal77}). 
\begin{thm}[Moser]\label{wow} Let $u \geq 0$ be a weak solution of $(\partial_{t} + H)u = 0$
in $Q_{2r}(x_{0}, t_{0})$. There exists $C > 0$, depending only on $n \geq 1$, such that
\begin{equation}\label{myman}
	\sup_{Q_{r/2}(x_{0}, t_{0})} |u(x,t)| \leq
		\biggl(\frac{C}{r^{n+2}}\iint_{Q_{2r/3}(x_{0},t_{0})} |u|^{2} \,dx\,dt\biggr)^{1/2}
\end{equation}
\end{thm} 
\begin{proof}[Sketch of proof]
We give a very precise reference because this fact is so basic for us. 
Suppose $Q_{2r}(x_{0}, t_{0}) = B_2(0) \times (0, 4)$. Note 
that because $V \geq 0$,
$u$ is a weak subsolution of $(\partial_{t} - \Delta)u=0$ in $Q_{2}(0,4)$. So subsolution
estimates for the heat equation apply. In particular, a slight modification
in the geometry of Moser's Theorem~1 in \cite{Mos64} 
(see especially pp.\ 124-125) establishes
$$
	\sup_{Q_{1/2}(0, 4)} |u(x,t)| \leq
		\biggl(C\iint_{Q_{2/3}(0,4)} |u|^{2} \,dx\,dt\biggr)^{1/2}
$$
Translation invariance of the heat equation then implies the lemma
with $r=1$, and the result for arbitrary $r > 0$ follows from invariance of the heat
equation under the scaling $x \to r x$, $t \to r^2 t$.
\end{proof}

\subsection{A Fefferman-Phong inequality}\label{fp}
Next we consider how to trade an $L^2$ 
norm bound like \eqref{myman} for an $L^2$ energy bound 
by introducing a $V$-dependent weight. 
What follows is the $p=2$ case of Auscher and Ben Ali's ``improved Fefferman-Phong inequality'' from \cite{AuBA07}.
\begin{thm}[Auscher, Ben Ali] \label{ifp} Suppose $V \in A_{\infty}$. Then there are constants $C > 0$ and
$\beta \in (0,1)$, depending only on $n  \geq 1$ and the \Ainf\ constant of $V$, such that for any cube
$Z = Z_{r}(x_{0})$ and $u \in C^{1}(\mathbb{R}^{n})$ one has
\begin{equation}
	\label{eq:ifp}
	\int_{Z}|\nabla u|^{2} + V\,|u|^{2}\,dx
		\geq C\frac{m_{\beta}(r^{2}\,\av_{Z} V)}{r^{2}}\int_{Z} |u|^{2}\,dx
\end{equation}
where $m_{\beta}(x) = x$ for $x \leq 1$ and $m_{\beta}(x) = x^{\beta}$ for $x \geq 1$. 
In particular, if $V \in A_{p}$, then one may take $\beta = \frac{2}{2 + n(p - 1)}$.
\end{thm}
The name comes from Fefferman and Phong's ``Main Lemma'' in \cite{Fef83} 
for polynomial potentials $V$ on $\mathbb{R}^n$, which concludes 
$\int_Z |\nabla u|^2 + V|u|^2\,dx \gtrsim R^{-2}\int_{Z}|u|^2\,dx$
for reasonable functions $u(x)$. As we will show, Auscher and Ben Ali's work provides just the
generalization we need to sharpen Moser's bound \eqref{myman} to include
effects of $V$. 

\subsection{A Caccioppoli inequality}\label{ci} Broadly speaking, a Caccioppoli inequality bounds the local energy of a weak solution 
to an elliptic or parabolic equation by its $L^2$ 
norm over a slightly larger set. This is what we need for the third ingredient
of our iteration. The version we need also appears as Lemma 3 in \cite{Kur00}; 
we state it here and provide a proof in Section~\ref{proof} for completeness.
\begin{lmm}\label{cacc}
Fix $\sigma \in (0,1)$. If $u$ is a weak solution to $(\partial_{t} + H)u = 0$ in
$Q_{2r}(x_{0}, t_{0})$, then there exists $C > 0$, depending only on $n \geq 1$, such that
\begin{multline*}
	\sup_{t_{0} - (\sigma r)^{2} \leq t \leq t_{0}} \int_{B(x_{0}, \sigma r)} |u(x,t)|^{2}\,dx\
		 +  \iint_{Q_{\sigma r}(x_{0}, t_{0})}\biggl(|\nabla u|^{2} + V\,|u|^{2}\biggr)\,dx\,ds \\
	\leq \frac{C}{(1-\sigma)^{2} r^{2}}\iint_{Q_{r}(x_{0},t_{0})} |u|^{2}\,dx\,dt
\end{multline*}
\end{lmm}

\section{Proof of The Upper Bound}\label{proof}

We hope that the reader has gained some intuition for how these three 
main ingredients might be combined, and provide the details next.

\begin{proof}[Proof of Theorem \ref{bus}] Fix $y \in \mathbb{R}^{n}$, and 
focus on the cylinder $Q_{r}(x,t)$ with $r = \sqrt{t/8}$. Write $u(\cdot,s) = p(\cdot,y,s)$ so that $u$ is a weak solution to $(\partial_{s} + H)u=0$ 
in $Q_{2r}(x,t)$. We will define an increasing sequence of cylinders that starts with $Q_{2/3r}(x,t)$. 
In particular, choose $k \in \mathbb{N}$ and define
$$ \rho_{j} = \frac{2}{3} + \biggl(\frac{j-1}{k}\biggr)\frac{1}{3} \qquad\qquad \text{for}\quad  j = 1,2,\ldots ,k+1.$$
These $\rho_{1},\ldots ,\rho_{k+1}$ are a sequence of $k$ scaling factors increasing from $\rho_{1}=2/3$ to
$\rho_{k+1} =1$. For each $j = 2,\ldots ,k+1$, also define nonnegative cutoff functions
$\chi_{j}(z) \in C_{0}^{\infty}(B(x, \rho_{j}r))$ and $\eta_{j}(s) \in C^{\infty}(\mathbb{R})$,
bounded by 1 and satisfying
\begin{align*}
	\quad \bullet \quad & \chi_{j} \equiv 1 \text{ on } B(x,\rho_{j-1} r),\
			|\nabla\chi_{j}| \leq \frac{Ck}{r} \\
	\quad \bullet \quad &  \eta_{j} \equiv 0 \text{ for } t \leq t_{0}-(\rho_{j}r)^{2},\
			\eta_{j} \equiv 1 \text{ for } t \geq t_{0} - (\rho_{j-1} r)^{2},\
			|\eta_{j}'| \leq \frac{Ck}{r^{2}}.		
\end{align*}
Note in particular that $\spt{\chi_{j}\eta_{j}} \subset B(x,r) \times [t_{0}-r^2, \infty)$.

Consider one of these cylinders 
$Q_{\rho_{j+1}r}(x,t)$, where $j = 1,\dots ,k$. Take the radius $r$ in Lemma~\ref{cacc} to be our
 $\rho_{j+1}r$; and take the scaling factor $\sigma$ in Lemma~\ref{cacc} to be  $\frac{\rho_{j}}{\rho_{j+1}}$. Then applying the Caccioppoli inequality,
$$
\iint_{Q_{\rho_{j+1}r}(x, t)}\biggl(|\nabla u|^2\chi_{j+1}^2\eta_{j+1}^2 +
	Vu^2\chi_{j+1}^2\eta_{j+1}^2\biggr)\,dz\,ds 
	 \leq \frac{Ck^2}{r^2}\iint_{Q_{\rho_{j+1}r}(x,t)}|u|^2\,dz\,ds.
$$
And from Cauchy's inequality
$$|\nabla(\eta_{j+1} u\chi_{j+1})|^2 \leq 2|\nabla u|^2\eta_{j+1}^2\chi_{j+1}^2
	+ 2|u|^2|\nabla\chi_{j+1}|^2\eta_{j+1}^2.$$
So using the bounds on $|\nabla\chi_{j+1}|$ and $\eta_{j+1}$, increasing $C$ as
necessary, we also have
\begin{multline}\label{five}
\int_{t-(\rho_{j+1}r)^2}^{t}\int_{B_{\rho_{j+1}r}(x)}\biggl(|\nabla(\eta_{j+1} u\chi_{j+1})|^2  +
	V|u|^2\chi_{j+1}^2\eta_{j+1}^2\biggr)\,dz\,ds \\
	 \leq \frac{Ck^2}{r^2}\iint_{Q_{\rho_{j+1}r}(x,t)}|u|^2\,dz\,ds.
\end{multline}

Now note that on the left-hand side of this inequality, we may apply the Fefferman-Phong inequality \eqref{eq:ifp} to the integral in the space 
directions. We do this on a cube containing the support
of $\chi_{j+1}$, namely $Z_{2r}(x)$, 
\begin{align*}
\int_{B_{\rho_{j+1}r}(x)}\biggl(|\nabla(\eta_{j+1} u\chi_{j+1})|^2 & +
	V|u|^2\chi_{j+1}^2\eta_{j+1}^2\biggr)\,dz \\
	&= \int_{Z_{2r}(x)}\biggl(|\nabla(\eta_{j+1} u\chi_{j+1})|^2  +
	V|u|^2\chi_{j+1}^2\eta_{j+1}^2\biggr)\,dz \\
	& \geq\frac{C}{r^{2}}m_{\beta}(r^2\,\av_{Z_{2r}(x)}V)
		\int_{B_{\rho_{j+1}r}(x)}|\eta_{j+1}u\chi_{j+1}|^2\,dz.
\end{align*}
Combined with \eqref{five}, this implies
$$
	\int_{t-(\rho_{j+1}r)^2}^{t} \frac{m_{\beta}(r^2\,\av_{Z_{2r}(x)}V)}{r^{2}}
		\int_{B_{\rho_{j+1}r}(x)}|\eta_{j+1}u\chi_{j+1}|^2\,dz\,ds 
	 \leq \frac{Ck^2}{r^2}\iint_{Q_{\rho_{j+1}r}(x,t)}|u|^2\,dz\,ds,
$$
and hence
$$
\iint_{Q_{\rho_{j+1}r}(x,t)}|\eta_{j+1}u\chi_{j+1}|^2\,dz\,ds 
	\leq \frac{Ck^{2}}{m_{\beta}(r^2\,\av_{Z_{2r}(x)}V)}
			\iint_{Q_{\rho_{j+1}r}(x,t)}|u|^2\,dz\,ds. $$

In other words, combining Lemma~\ref{cacc} and Theorem~\ref{ifp} 
lets us directly relate the $L^2$ norm of $u$ at the the $\rho_{j}$ scaling to the
$\rho_{j+1}$ scaling,
\begin{align*}
	\iint_{Q_{\rho_{j}r}(x,t)}|u|^{2}\,dz\,ds & 
	\leq \iint_{Q_{\rho_{j+1}r}(x,t)}|\eta_{j+1}u\chi_{j+1}|^2\,dz\,ds  \\
		& \leq \frac{Ck^{2}}{m_{\beta}(r^2\,\av_{Z_{2r}(x)}V)}
			\iint_{Q_{\rho_{j+1}r}(x,t)}|u|^2\,dz\,ds.
\end{align*}
Starting with $j=1$ and iterating this relation $k$ times yields 
$$\iint_{Q_{2r/3}(x,t)} |u|^{2}\,dz\,ds \leq
	\frac{C^k k^{2k}}{m_{\beta}(r^2\,\av_{Z_{2r}(x)}V)^k}
			\iint_{Q_{r}(x,t)}|u|^2\,dz\,ds $$
Which substituted into Moser's estimate \eqref{myman} yields
\begin{equation}\label{anes}
\sup_{(z,s) \in Q_{r/2}(x, t)}|u| \lesssim
		\frac{C^{k/2} k^{k}}{m_{\beta}(r^2\,\av_{Z_{2r}(x)}V)^{k/2}}
			\biggl(\frac{1}{r^{n+2}}\iint_{Q_{r}(x,t)}|u|^2\,dz\,ds\biggr)^{1/2}  
\end{equation}
with the suppressed constant independent of $k$. By \cite[Proposition 3.1]{BoRa13h}, it follows immediately that with $c_1$ in that proposition as $(eC)^{-1}$,
$$
\sup_{(z,s) \in Q_{r/2}(x, t)}|u| \leq c_{0}
	\exp\left\{-c_{1}m_{\beta}(r^2\,\av_{Z_{2r}(x)}V)^{1/2}\right\} 
	\biggl(\frac{1}{r^{n+2}}\iint_{Q_{r}(x,t)}|u|^2\,dz\,ds\biggr)^{1/2},
$$
and recalling that $r = \sqrt{t/8}$ and $u(\cdot,s) = p(\cdot,y,s)$
\begin{equation}\label{some}
p(x,y,t) \leq c_{0}
	\exp\left\{-c_{1}m_{\beta}(t\,\av_{Z_{\sqrt{t/2}}(x)}V)^{1/2}\right\} 
	\biggl(\frac{1}{t^{(n+2)/2}}\iint_{Q_{\sqrt{t/8}}(x,t)}|p|^2\,dz\,ds\biggr)^{1/2}.
\end{equation}
Because $A_{\infty}$ potentials are doubling, in \eqref{some} 
we can now replace the average of $V$ over $Z_{\sqrt{t/2}}(x)$ with its average over $Z_{\sqrt{t}}(x)$, scaling $c_{1}$ appropriately.  

The final step is to incorporate the Gaussian bound on the heat kernel
\begin{equation}\label{lap}
p(x,y,t) \lesssim t^{-n/2}\exp(-|x-y|^2/4t).
\end{equation}
In particular, if $|x-y| \approx \sqrt{t}$, we lose nothing by using $p \lesssim t^{-n/2}$ inside the integral 
in \eqref{some}. Because $|Q_{r}| \approx  r^{n+2}$, this gives
\begin{equation}\label{hhh}
	p(x,y,t) \leq \frac{c_{0}}{t^{n/2}}\exp\left\{-c_{1}m_{\beta}(t\,\av_{Z_{\sqrt{t}}(x)}V)^{1/2}\right\}.
\end{equation}
On the other hand, if $\sqrt{t} \ll |x-y|$ then $p \lesssim t^{-n/2}$ is a very poor estimate. We would
be better off just using \eqref{lap} directly. So the upper bound of the theorem is a compromise that follows 
from writing 
$$p(x,y,t) = p(x,y,t)^{1/2}\cdot p(x,y,t)^{1/2}$$ 
and then applying \eqref{lap} to the first term in the product, \eqref{hhh} to the second.
\end{proof}  

\begin{proof}[Proof of Lemma \ref{cacc}] First, the argument in \S9 of \cite{Aro67} allows us to assume that $u$ has a
strong derivative $\partial_{t}u \in L^{2}(Q_{2r}(x_{0},t_{0}))$. Now choose nonnegative cutoff functions
$\chi(x) \in C_{0}^{\infty}(B(x_{0},r))$ and $\eta(s) \in C^{\infty}(\mathbb{R})$, bounded above by
1 and satisfying
\begin{align*}
	\bullet \quad & \chi(x) \equiv 1 \text{ on } B(x_{0},\sigma r),\
			|\nabla\chi(x)| \leq \frac{C}{(1-\sigma)r}; \\
	\bullet \quad &  \eta(s) \equiv 0 \text{ for } s \leq t_{0}-r^{2},\
			\eta(s) \equiv 1 \text{ for } s \geq t_{0} - (\sigma r)^{2},\
			|\eta'(s)| \leq \frac{C}{(1-\sigma)r^{2}}	.	
\end{align*}
Fix $t \in [t_{0}-(\sigma r)^{2}, t_{0}].$ Note that the test function $\eta^{2}(s)\chi^{2}(x)u(x,s)$ belongs
to the class $\mathcal{C}$ specified in Definition~\ref{weakly}. Hence we may use
this function for $\phi(x,s)$ in \eqref{ws}. This yields, since $\eta(t) = 1$,
\begin{multline}\label{subs}
   \int_{B(x_{0},r)}u^2\chi^2\,dx - 
   \int_{t_{0}-r^2}^{t}\int_{B(x_{0},r)}(u^2(2\eta\,\eta')\chi^2 + (u\,\partial_{s}u)\eta^2\chi^2)\,dx\,ds \\
   + \int_{t_{0}-r^2}^{t}\int_{B(x_{0},r)} ((\nabla u\cdot\nabla\chi^2)\eta^2u + 
   |\nabla u|^2\eta^2\chi^2 + Vu^2\eta^2\chi^2)\,dx\,ds = 0.
\end{multline}
Note that the second integral in \eqref{subs} may be written as
$$ 
 \int_{t_{0}-r^2}^{t}\int_{B(x_{0},r)}\frac{1}{2}\partial_{s}(u^2\eta^2\chi^2)\,dx\,ds
  + \int_{t_{0}-r^2}^{t}\int_{B(x_{0},r)}(u^2(\eta\,\eta')\chi^2\,dx\,ds.
 $$
 And by the bounded convergence theorem we may interchange integration and differentiation in the 
first term above, so because $\eta(t_{0} - r^2) = 0$ it follows
 $$  \int_{t_{0}-r^2}^{t}\int_{B(x_{0},r)}\frac{1}{2}\partial_{s}(u^2\eta^2\chi^2)\,dx\,ds
  = \frac{1}{2}\int_{B(x_{0},r)}u^2\chi^2\,dx.$$
Substituting these observations into \eqref{subs}, we obtain
\begin{multline}\label{ws:int}
	\frac{1}{2}\int_{B(x_{0},r)} u^{2}\chi^{2}\,dx +
		\int_{t_{0}-r^2}^{t}\int_{B(x_{0},r)} \bigl(|\nabla u|^{2} \eta^2 \chi^2 + V u^2\eta^2\chi^2\bigr)\,dx\,ds \\
	= \int_{t_{0}-r^2}^{t}\int_{B(x_{0},r)} \bigl(u^2\chi^2\eta\,\eta' -
		(\nabla u \cdot \nabla\chi^2)\, \eta^2 u\bigr)\,dx\,ds
\end{multline}
Since $t$ was arbitrary and $V \geq 0$, we next conclude
$$
	\sup_{t_{0}-(\sigma r)^{2} \leq t \leq t_{0}} \frac{1}{2}
		\int_{B(x_{0},r)} u^2\chi^{2}\,dx \leq
	\iint_{Q_{r}(x_{0}, t_{0})} \bigl(u^2|\eta'| +
	|\nabla u|\chi\,\eta^2\,|u| |\nabla \chi|\bigr)\,dx\,ds.
$$	
In the second term of the righthand integral we apply Cauchy's inequality \linebreak
$ab \leq \frac{a^2}{2} + \frac{b^2}{2}$,
using $a = |\nabla u|\chi$, $b = |u| |\nabla \chi|$. It follows
\begin{multline*}
	\sup_{t_{0}-(\sigma r)^{2} \leq t \leq t_{0}} \frac{1}{2}
		\int_{B(x_{0},r)} u^2\chi^{2}\,dx \leq
	\iint_{Q_{r}(x_{0}, t_{0})} u^2|\eta'|\,dx\,ds\ + \\
	\qquad \frac{1}{2}\iint_{Q_{r}(x_{0}, t_{0})} \chi^2\eta^2|\nabla u|^2\,dx\,ds
	 + \frac{1}{2}\iint_{Q_{r}(x_{0}, t_{0})}|u|^2|\nabla\chi|^2\eta^2\,dx\,ds
\end{multline*}
And from the bounds on $|\nabla \chi|$, $\eta$, and $|\eta'|$,
\begin{multline}\label{two}
	\sup_{t_{0}-(\sigma r)^{2} \leq t \leq t_{0}}
		\int_{B(x_{0},r)} u^2\chi^{2}\,dx \\ \leq 
	\frac{C}{(1-\sigma)^2r^2}\iint_{Q_{r}(x_{0}, t_{0})}u^2\,dx\,ds\ + 
 \iint_{Q_{r}(x_{0}, t_{0})}\chi^2\eta^2|\nabla u|^2\,dx\,ds
\end{multline}

To complete the proof we make another application of \eqref{ws:int},
this time with $t = t_{0}$. The positivity of the leftmost integral and the bounds on
$\chi$, $\eta$, and $|\eta'|$ yield
\begin{multline*}
	\iint_{Q_{r}(x_{0},t_{0})}\biggl(|\nabla u|^{2}\chi^2\,\eta^2 +
		Vu^2\chi^2\eta^2\biggr)\,dx\,ds \\
		\leq \frac{C}{(1-\sigma)r^2}\iint_{Q_{r}(x_{0},t_{0})}
		u^2 \,dx\,ds +
		\iint_{Q_{r}(x_{0},t_{0})} |\nabla u|\chi\,\eta^2\,|u| |\nabla \chi|\,dx\,ds
\end{multline*}
And the above use of Cauchy's inequality gives, after
rearranging, absorbing terms, and possibly increasing $C$,
\begin{multline}\label{three}
	\iint_{Q_{r}(x_{0}, t_{0})}\chi^2\eta^2|\nabla u|^2\,dx\,ds
	 + \iint_{Q_{r}(x_{0},t_{0})}Vu^2\chi^2\eta^2\,dx\,ds \\
	 \leq \frac{C}{(1-\sigma)^2r^2}\iint_{Q_{r}(x_{0}, t_{0})}u^2\,dx\,ds
\end{multline}
Combining \eqref{two} and \eqref{three} and restricting the lefthand
integrals to where the cutoff functions are unity yields the lemma. 
Note that \eqref{three} is all that is needed for the proof
of Theorem~\eqref{bus}.
\end{proof}

\begin{proof}[Proof of Theorem \ref{explicit}] 
It suffices to treat the case $a_{0} = 0$, because
\begin{equation}\label{aaa}
\left[\partial_{t} - \Delta + (a_{2}x^2 + a_{1}x + a_{0})\right]u =
 	e^{-a_{0}t} \left[\partial_{t} - \Delta + (a_{2}x^2 + a_{1}x)\right]e^{a_{0}t}u.
\end{equation}
So if $p_{0}(x,y,t)$ is the heat kernel for the operator with potential $V(x) = \sum_{i=1}^2 a_{i}x^{i}$,
then $e^{-a_{0}t}p_{0}(x,y,t)$ is directly checked to be the heat kernel when $V(x) = \sum_{i=0}^2 a_{i}x^{i}$.

Several approaches to the calculation are possible. Interpreting
$p(x,y,t)$ as the transition probability of a system from state $x$ to state $y$ in time $t$, the 
kernel is determined by a certain path integral of the Lagrangian given by $(-\Delta + V)$.
For quadratic $V$ this path integral can then be computed using the van Vleck determinant (see
\cite{Vis93} for details.) Another possibility is to begin with the Mehler kernel of the harmonic
oscillator (see again \cite{Dav89}) and study the behavior of this kernel under appropriate 
scalings and translations of the harmonic oscillator. Our method, requiring rather less theory 
than either of the above, is to take from \cite{Bea99} the \emph{ansatz}
$$p(x,y,t) = \phi(t)\exp\left\{-\frac{1}{2}\left(\alpha(t)x^2 + \gamma(t)y^2 + 2\beta(t)xy\right)\right\}
	\exp\left\{-\mu(t)x - \nu(t)y\right\}.$$
We then simply attempt to enforce on this ansatz the two conditions
\begingroup
\begin{equation}\label{hkeq}
	\left\{\begin{array}{ll}
			(\partial_{t} + H)p(\cdot,y,t) = 0 & \textrm{ on } \mathbb{R}\times(0,\infty)  \\
			\displaystyle \lim_{t \to 0^{+}}p(x,y,t) = \delta(x - y) & \textrm{ in } L^1(\mathbb{R})
		\end{array}\right.
\end{equation}
\endgroup
The differential condition in (\ref{hkeq}) yields a system of six ODE's in $t$.
\begin{align}
	\label{eq1} \alpha' &= -2\alpha^2 + 2a_{2} \\
	\label{eq2} \beta' &= -2\alpha \beta \\
	\label{eq3} \gamma' &= -2\beta^2 \\
	\label{eq4 }\mu' &= a_{1} - 2\mu\alpha \\
	\label{eq5} \nu' &= -2\mu\beta \\
	\label{eq6} \phi'/\phi &= -\alpha + \mu^2.
\end{align}
These equations are solvable, in order, by elementary methods. The
constants of integration must be chosen according to the second condition
in (\ref{hkeq}). For $\beta(t)$ this is actually easy to do, but in solving the
remaining equations we set constants of integration to zero and identify 
their true values only after the functional form of $p(x,y,t)$ is known. Also, the singularity condition of \eqref{hkeq} is equivalent to two properties. First,
$$\displaystyle \lim_{t \to 0^+}p(x,y,t) = 0$$ 
whenever $x \neq y$; and, second,
$$ \displaystyle \lim_{t \to 0^+} \int_{\mathbb{R}} p(x,y,t)\,dy = 1 $$
for any $x \in \mathbb{R}$.

The result follows from elementary, though tedious, computations.
\end{proof}

\section{A Lower Bound for $p$ when $V \in RH_\infty$}\label{result2}

For $V$ which belongs to a local Kato class and decays or is $L^\infty$ bounded
at infinity, Zhang and Zhao \cite{ZhZh00} proved the attractive lower bound
\[
	p(x,y,t) \geq \left\{\begin{array}{lr}
		\frac{c_{1}}{t^{n/2}}e^{-c_{2}K_{V^{+}}(t)}
			& \quad |x-y|^{2} \leq t \\
		\frac{c_{1}}{t^{n/2}}e^{-c_{2}\frac{|x-y|^{2}}{t}[1 + K_{V^{+}}(\frac{t^{2}}{|x-y|^{2}})]}
			& \quad |x-y|^{2} \geq t
	\end{array}\right.
\]
where 
$$K_{V}(t) = \sup_{x}\int_{0}^{t}\int_{\mathbb{R}^{n}}\frac{1}{(t-s)^{n/2}}e^{-c\frac{|x-y|^{2}}{t-s}}
|V(y)|\,dy\,ds.$$
But to our knowledge, the case with $V$ unbounded near infinity has 
not received any attention in the literature. So here we provide a lower bound 
for $p$ whose dependence on $V$ is analogous to the $V$-dependent decay of Theorem~\ref{bus}. 
However, it applies only for $V \in RH_\infty$, the most tractable reverse \HOL\ class.

\begin{thm}\label{lbounds} If $V \in RH_{\infty}$, the heat kernel of the Schr\"{o}dinger
operator $H = -\Delta + V$ on $L^{2}(\mathbb{R}^{n})$ satisfies with $0 < \kappa < 1$
fixed
\begin{equation}\label{convenient2}
	p(x,y,t) \geq \left\{\begin{array}{lr} \displaystyle
		\frac{c_{0}}{t^{n/2}} \exp\{-c_{1}t\,\av_{Z_{\sqrt{t}}(x)}V\}
			& \quad |x-y| < \kappa\sqrt{t} \vspace{.1in} \\ \displaystyle 
			\frac{c_{0}}{t^{n/2}}e^{-c_{3}\frac{|x-y|^{2}}{t}}\exp\bigl\{-c_{1}t
		(c_{2}^{\frac{|x-y|^{2}}{t}}\av_{Z_{t/|x-y|}(x)}V)\bigr\}
			& \quad |x-y| \geq \kappa\sqrt{t}
	\end{array}\right.
\end{equation}
for some constants $c_{i} > 0$ for $i =0,1,2,3$.
\end{thm}

The idea of the proof is to establish a bridge between $p$ and the heat kernel of an appropriate
Dirichlet Laplacian, where van den Berg's results in \cite{Ber90} can be 
applied. We will do this in Section~\ref{proof2} using the semigroup property of $p(x,y,t)$, a
parabolic maximum principle, and a lemma that relates the averages of a doubling measure over
nested cubes. Let us first review our technical devices and van den Berg's estimates.

\subsection{The semigroup property}

\begin{lmm}\label{semi} Let $p(x,y,t)$ be the heat kernel of a \SCH\ operator $H$ on
$L^2(\mathbb{R}^{n})$ with locally integrable nonnegative potential. Then 
\begin{equation*}
	p(x,y,t+s) = \int_{\mathbb{R}^{n}} p(x,z,t)\,p(z,y,s)\,dz 
\end{equation*} 
for all $x,y \in \mathbb{R}^{n}$ and $s,t > 0$. \qed
\end{lmm}

This restates the standard fact that $e^{-Ht}$ is a semigroup, but in the precise form we
will need it. For an off-diagonal 
estimate of $p(x,y,t)$, we will invoke Lemma~\ref{semi} repeatedly to write $p(x,y,t)$ as an
iterated integral of many ``copies'' of itself at earlier times. Our on-diagonal bounds 
will then apply to these copies when they are restricted to appropriately small regions in space.

\subsection{A maximum principle}

The following maximum principle will connect the heat kernel of $H = -\Delta + V$ to that
of an appropriate Dirichlet Laplacian. Note that we need the local boundedness of $V$ 
implied by membership in $RH_{\infty}$. 

\begin{thm}\label{mymax}
Suppose $V \geq 0$ is bounded in a cylinder $Q = Q_{r}(x_{0},t_{0})$ and $u \in C(\overline{Q})$
is a weak solution of $(\partial_{t} + H)u = 0$ in $Q$. Then 
$$\sup_{Q} u \leq \sup_{\partial Q}u_{+} \quad \textrm{ and } \quad \inf_{Q}u \geq \inf_{\partial Q}u_{-}. $$
If $u$ is only a weak supersolution of the same equation in $Q$, then we still have  the 
second conclusion. \qed
\end{thm}
Since we cannot assume $u$ is a classical solution in $Q$, the proof
is rather involved and best accomplished through functional analytic machinery.  
Details may be found in \cite{KlNa06}.

\subsection{A lemma for doubling measures}

The following useful lemma comes from Christ \cite{Chr91}, and allows us to compare
the averages of $V$ over two cubes whose centers are some distance from each other.

\begin{lmm}\label{doubling}
For any doubling measure $\omega$ on $\mathbb{R}^{n}$, there exist positive $C < \infty$
and $\epsilon < 1$ such that for any cubes $Z' \subset Z$
$$\int_{Z'}\,d\omega \leq C\left(\frac{|Z'|}{|Z|}\right)^{\epsilon}\int_{Z}\,d\omega$$
where e.g., $|Z|$ denotes the Euclidean measure of $Z$. \qed
\end{lmm}

\subsection{Estimates on the Dirichlet heat kernel}
Now we come to van den Berg's results on Dirichlet heat kernels. Their statement differs somewhat for 
the $n=1$ and $n\geq2$ cases. In the latter case we need the following definition. 

\begin{defn}\label{crazy}
Fix an open set $D \subset \mathbb{R}^n$, where $n \geq 2$. Given $\epsilon > 0$, 
let $D_{\epsilon}$ be the points in $D$ at least distance $\epsilon$ from the boundary; 
and let $d_{\epsilon}(x,y)$ for $x,y \in D$ be the infimum of lengths of arcs in $D_{\epsilon}$ 
with endpoints $x$ and $y$. When $d_{\epsilon}(x,y) < \infty$, let $\gamma_{\epsilon}$
be a minimal geodesic from $x$ to $y$ and define
$$\alpha(\gamma_{\epsilon}) = \int_{s \colon \gamma_{\epsilon}(s) \in D_{\epsilon}}
	\biggl|\frac{d^{2}\gamma_{\epsilon}(s)}{ds^2}\biggr|\,ds.$$
\end{defn}

\begin{thm}[van den Berg]\label{berg} Suppose $D$ is an open set in $\mathbb{R}^{n}$ 
with $n \geq 2$. Given $\epsilon > 0$, $0 < \delta \leq \epsilon$, $x \in D$, $y \in D$ such 
$d_{\epsilon}(x,y) < \infty$, it holds for all $t > 0$
$$\Gamma_{D}(x,y,t) \geq \frac{C}{t^{n/2}} e^{-\frac{\pi^2 n^2 t}{4\epsilon^2}} 
	\exp\left\{-\frac{d_{\epsilon}(x,y)^2\bigl(1 + 2\delta\alpha(\gamma_{\epsilon})
	d_{\epsilon}(x,y)\bigr)}{4t}\right\} $$
where $\Gamma_{D}$ is the heat kernel of $-\Delta$ on $D$ with Dirichlet boundary 
conditions and $C < 1$ is a positive constant depending only on $n$.  \qed
\end{thm}
When using Theorem~\ref{berg} we will always choose $D$ to be a ball, so that $D_{\epsilon}$ is also a ball and
$$\gamma_{\epsilon}(s) = x + \frac{s}{|y - x|}(y-x).$$
Hence $\alpha(\gamma_{\epsilon})$ will always vanish in our applications.

When $n=1$, van den Berg obtained a lower bound on $\Gamma_{D}$ from
ingenious use of a special function identity and the eigenfunction expansion of the 
Dirichlet heat kernel on an interval. Namely, 

\begin{prop}[van den Berg]\label{eigen} Suppose $D \subset \mathbb{R}$ is an interval, and for 
some $x < y$ and $\epsilon > 0$ we have $(x - \epsilon, y + \epsilon) \subset D$. Then
for all $t > 0$
$$\Gamma_{D}(x,y,t) \geq \frac{C}{t^{1/2}}e^{-\frac{|x-y|^{2}}{4t}}(1 - 2e^{-\frac{\epsilon^2}{t}})$$
where $\Gamma_{D}$ is the heat kernel of $-\Delta$ on $D$ with Dirichlet boundary 
conditions and $C < 1$ is a positive constant depending only on $n$. 
\end{prop} 

 \section{Proof of the Lower Bound}\label{proof2}
 
 \begin{proof}[Proof of Theorem \ref{lbounds}]
First suppose $|x - y| < \frac{1}{8}\sqrt{t}$. We consider the ball $B = B_{\sqrt{t}}(x)$.
Let $H_{B}$ be the restriction of the operator $H$ to $B$ with Dirichlet boundary conditions;
and let $p_{B}(x,y,t)$ be the associated heat kernel. Note that 
$u(\cdot, t) = p(\cdot,y,t) - p_{B}(\cdot,y,t)$ is a weak solution to $(\partial_{t} + H)u = 0$ on $B \times (0,\infty)$,
for any $y \in B$. And because $p_{B}(\cdot,y,t)$ vanishes on $\partial B$, $u$ is nonnegative on the 
boundary, implying by the maximum principle that
\begin{equation}\label{lb1}
p(x,y,t) \geq p_{B}(x,y,t)\qquad \textrm{in } B\times B\times(0,\infty)
\end{equation}
since the choice of $y \in B$ was arbitrary.

Now we use again the hypothesis that $V \in RH_{\infty}$. With $C > 0$ the $RH_{\infty}$ constant of $V$,
we have for $M = C \av_{Z_{2\sqrt{t}}(x)}V$ that $V \leq M$ in $B$. Let $H_{D}^{M}$ be the operator $(-\Delta + M)$ 
restricted to $B$ with Dirichlet boundary conditions; so its heat kernel is just
$e^{-Mt}\Gamma_{D}$, where $\Gamma_{D}$ is the heat kernel of the Dirichlet 
Laplacian on $B$. Now for $y \in B$ set
 $$w(x,t) = p_{B}(x,y,t) - e^{-Mt}\,\Gamma_{D}(x,y,t).$$
 Then $w \equiv 0$ on $\partial B \times (0,\infty)$, and inside $B$ we have for any $t > 0$
$$
 	(\partial_{t} + H_{D}^{M})w = (\partial_{t} - \Delta + M)p_{B} 
		= (M - V)p_{B} 
		\geq 0.
 $$
 So $w$ is a supersolution of $(\partial_{t} + H_{B}^{M})$ in the cylinder 
 $Q = B \times (0,\infty)$, vanishing on the boundary, and by the maximum principle satisfies
 $$p_{B}(x,y,t) \geq e^{-Mt}\,\Gamma_{D}(x,y,t)\qquad \textrm{in } B\times B\times(0,\infty)$$
 since $y \in B$ was arbitrary. 
 
 Applying either Proposition~\ref{eigen} or Theorem~\ref{berg}
 with $\epsilon = \frac{7}{8}\sqrt{t}$, we obtain from the preceding inequality and \eqref{lb1}
 that
 \begin{equation}\label{ondiag}
 p(x,y,t) \geq \frac{c_{0}}{t^{n/2}} \exp\left\{-c_{1}t\,\av_{Z_{2\sqrt{t}}(x)}V\right\}
 \end{equation}
 where $c_{0} < 1$ is a positive constant depending only on $n$, and $c_{1}$ is just the
 $RH_{\infty}$ constant of $V$. Because $V$ is doubling, we may also increase $c_{1}$
 and replace the cube $Z_{2\sqrt{t}}(x)$ with $Z_{\sqrt{t}}(x)$. These on-diagonal bounds 
 conclude the first part of the
 proof. 
 
 Off-diagonal bounds come next. Assume $|x - y| \geq \frac{1}{8}\sqrt{t}$.
We begin by considering the line segment from $x$ to $y$ given by
$$l(s) = x + s\frac{y-x}{|y-x|}, \qquad s \in \left[0,|y-x|\right].$$
We will partition this segment by a sequence of $M+1$ points $\{x_{i}\}_{i=0}^{M}$;
the sequence is determined by the requirement that $|x_{i} - x_{i+1}| = \frac{|y-x|}{M}$,
where $M$ is the smallest integer satisfying
\begin{equation}\label{aria}
\frac{|y-x|}{M} < \frac{1}{16}\sqrt{\frac{t}{M}} \Leftrightarrow \frac{256|y-x|^2}{t} < M
\end{equation}
Now directly from Lemma~\ref{semi} we have
$$	p(x,y,t) = \int_{\mathbb{R}^{n}}
		p(x,z_{1},t/M)p(z_{1},y,(M-1)t/M)\,dz_{1} $$
and applying the semigroup property in the same way to the right-most integrand, $(M-1)$ times,
we get an iterated integral
\begin{equation*}
p(x,y,t) = \int_{\mathbb{R}^{n}}\cdots\int_{\mathbb{R}^{n}}
			p(x,z_{1},t/M)\cdots p(z_{M-1},y,t/M)\,dz_{1}\cdots dz_{M-1}
\end{equation*}
Upon restricting each $dz_{i}$ integral to $Z_{i} = Z_{\sigma\sqrt{t/M}}(x_{i})$ with $0 < \sigma < 1$ such that
$$z_{i} \in Z_{i} \textrm{ and }  z_{i+1} \in Z_{i+1} \Rightarrow |z_{i} - z_{i+1}| < \frac{1}{8}\sqrt{t/M}$$
 we then obtain
\begin{equation}\label{intlb}
p(x,y,t) \geq \int_{Z_{1}}\cdots\int_{Z_{M-1}}
			p(x,z_{1},t/M)\cdots p(z_{M-1},y,t/M)\,dz_{1}\cdots dz_{M-1} 
\end{equation}
And now our on-diagonal lower bounds apply to each term in the integrand. 

That is, we have for each $i = 0,\ldots ,M-1$ that
$$p(z_{i},z_{i+1},t/M) \geq c_{0}\biggl(\frac{M}{t}\biggr)^{n/2} 
	\exp\left\{-c_{1}\frac{t}{M}\, \av_{Z_{\sigma\sqrt{t/M}}(z_{i})} V\right\}.$$
To assimilate these into a single lower bound for $p(x,y,t)$, we use Lemma~\ref{doubling}. 
In particular we see that 
\begin{align*}
  \int_{Z_{\sigma\sqrt{t/M}}(z_{i})}V &\leq C\left(\frac{1}{2^{n}}\right)^{\epsilon}
  	\int_{Z_{2\sigma\sqrt{t/M}}(x_{i})}V \\
	& \leq C\int_{Z_{i}}V
\end{align*}
and iterating this inequality up to $M$ times (if $i = M-1$) we may even conclude
$$\int_{Z_{\sigma\sqrt{t/M}}(z_{i})}V \leq C^{M} \int_{Z_{0}}V = C^{M} \int_{Z_{\sigma\sqrt{t/M}}(x)} V.$$
So in fact we have a lower bound, uniform in $i$, of
$$p(z_{i},z_{i+1},t/M) \geq c_{0}\biggl(\frac{M}{t}\biggr)^{n/2} 
	\exp\left\{-c_{1}\frac{t}{M}\, C^{M} \av_{Z_{\sigma\sqrt{t/M}}(x)} V\right\}$$
It now just remains to apply this to each term in the integrand of \eqref{intlb}.

This yields
\begin{align*}p(x,y,t) &\geq \prod_{i=0}^{M-1}c_{0}\biggl(\frac{M}{t}\biggr)^{n/2}
	\exp\left\{-c_{1}\frac{t}{M}C^{M}\, \av_{Z_{\sigma\sqrt{t/M}}(x)} V\right\}
	\cdot\prod_{i=1}^{M-1}|Z_{\sigma\sqrt{t/M}}(x_{i})| \\
    	&\geq \frac{\sigma^{-1}}{t^{n/2}} M^{n/2}(\sigma c_{0})^{M}
	\exp\left\{-c_{1}t\bigl(C^{M}\, \av_{Z_{\sigma\sqrt{t/M}}(x)} V\bigr)\right\}.
\end{align*}
Because $c_{0} < 1$ and $\sigma < 1$, the factor $M^{n/2}(\sigma c_{0})^{M}$
gives exponential decay in $M$;  
and by \eqref{aria}, $M$ is comparable to $\frac{|x-y|^2}{t}$. Increasing constants as
necessary and using Christ's lemma to replace $M$ with $\frac{|x-y|^2}{t}$
 yields \eqref{convenient2} with $\kappa = 1/8$. 
\end{proof}

\bibliographystyle{alpha}
\bibliography{mybib8-28-15}

\end{document}